\documentclass[12pt]{amsart}
\usepackage{amsmath,amsthm,latexsym,amscd,amsbsy,amssymb,amsfonts,fleqn,leqno,
euscript, pb-diagram}
\usepackage[T1]{fontenc}

\numberwithin{equation}{section}

\newtheorem{thm}{Theorem}[section]
\newtheorem{pro}[thm]{Proposition}

\newtheorem{cor}[thm]{Corollary}

\begin{document}


\title[A selection theorem for set-valued maps into normally supercompact spaces]
{A selection theorem for set-valued maps into normally supercompact spaces}

\author{V. Valov}
\address{Department of Computer Science and Mathematics,
Nipissing University, 100 College Drive, P.O. Box 5002, North Bay,
ON, P1B 8L7, Canada} \email{veskov@nipissingu.ca}

\date{\today}
\thanks{The author was partially supported by NSERC
Grant 261914-13.}




\begin{abstract}
The following selection theorem is established:\\ Let $X$ be a compactum possessing a binary
normal subbase $\mathcal S$ for its closed subsets. Then every set-valued $\mathcal S$-continuous
map $\Phi\colon Z\to X$ 
with closed $\mathcal S$-convex values, where $Z$ is an arbitrary space,
has a continuous single-valued selection.
More generally, if $A\subset Z$ is closed and any map from $A$ to
$X$ is continuously extendable to a map from $Z$ to $X$, then every selection
for $\Phi|A$ can be extended to a selection for $\Phi$.

This theorem implies that if $X$ is a $\kappa$-metrizable (resp., $\kappa$-metrizable and connected) compactum with
a normal binary closed subbase $\mathcal S$,
then every open $\mathcal S$-convex surjection $f\colon X\to Y$ is a zero-soft (resp., soft) map.
Our results provide
some generalizations and
specifications of Ivanov's results (see \cite{i1}, \cite{i2}, \cite{i3})
concerning superextensions of $\kappa$-metrizable compacta.
\end{abstract}

\maketitle

\markboth{}{Selections and $\kappa$-metrizable compacta}



\section{Introduction}
In this paper we assume that all topological spaces are Tychonoff
and all single-valued maps are continuous.

Recall that supercompact spaces and superextensions were introduced
by de Groot \cite{dg}. A space  is {\em supercompact} if it
possesses a binary subbase for its closed subsets. Here, a
collection $\mathcal S$ of closed subsets of $X$ is {\em binary}
provided any linked subfamily of $\mathcal S$ has a non-empty
intersection (we say that a system of subsets of $X$ is {\em linked}
provided any two elements of this system intersect). The
supercompact spaces with binary {\em normal} subbase will be of
special interest for us. A subbase $\mathcal S$ which is closed both
under finite intersections and finite unions is called normal if for
every $S_0,S_1\in\mathcal S$ with $S_0\cap S_1=\varnothing$ there
exists $T_0,T_1\in\mathcal S$ such that $S_0\cap
T_1=\varnothing=T_0\cap S_1$ and $T_0\cup T_1=X$. A space $X$
possessing a binary normal subbase $\mathcal S$ is called {\em
normally supercompact} \cite{vm} and will be denoted by $(X,\mathcal
S)$.

The {\em superextension} $\lambda X$ of $X$ consists of all maximal
linked systems of closed sets in $X$. The family
$$U^+=\{\eta\in\lambda X:F\subset U\hbox{~}\mbox{for
some}\hbox{~}F\in\eta\},$$ $U\subset X$ is open, is a subbase for
the topology of $\lambda X$. It is well known that $\lambda X$ is
normally supercompact. Let $\eta_x$, $x\in X$, be the maximal linked
system of all closed sets in $X$ containing $x$. The map
$x\to\eta_x$ embeds $X$ into $\lambda X$. The book of van Mill
\cite{vm} contains more information about normally supercompact
space and superextensions, see also Fedorchuk-Filippov's book
\cite{f1}.

If $\mathcal S$ is a closed subbase for $X$ and $B\subset X$, let
$I_{\mathcal S}(B)=\bigcap\{S\in\mathcal S:B\subset S\}.$ A subset
$B\subset X$ is called {\em $\mathcal S$-convex} if for all $x,y\in
B$ we have $I_{\mathcal S}(\{x,y\})\subset B$. An {\em $\mathcal
S$-convex map} $f\colon X\to Y$ is a map whose fibers are $\mathcal
S$-convex sets. A set-valued map $\Phi\colon Z\to X$ is said to be
{\em $\mathcal S$-continuous} provided for any $S\in\mathcal S$ both
sets $\{z\in Z:\Phi(z)\cap (X\backslash S)\neq\varnothing\}$ and
$\{z\in Z:\Phi(z)\subset X\backslash S\}$ are open in $Z$.

\begin{thm}
Let $(X,\mathcal S)$ be a normally supercompact space and $Z$ an
arbitrary space. Then every $\mathcal S$-continuous set-valued map
$\Phi\colon Z\to X$ has a single-valued selection provided all
$\Phi(z)$, $z\in Z$, are $\mathcal S$-convex closed subsets of $X$.
More generally, if $A\subset Z$ is closed and every map from $A$ to
$X$ can be extended to a map from $Z$ to $X$, then every selection
for $\Phi|A$ is extendable to a selection for $\Phi$.
\end{thm}

\begin{cor}
Let $\Phi\colon Z\to X$ be an $\mathcal S$-continuous set-valued map
such that each $\Phi(z)\subset X$ is closed, where $X$ is a space
with a binary normal closed subbase $\mathcal S$ and $Z$ arbitrary.
Then the map $\Psi\colon Z\to X$, $\Psi(z)=I_{\mathcal
S}\big(\Phi(z)\big)$, has a continuous selection.
\end{cor}

A map $f\colon X\to Y$ is invertible if for any space $Z$ and a map
$g\colon Z\to Y$ there exists a map $h\colon Z\to X$ with $f\circ
h=g$. If $X$ has a closed subbase $\mathcal S$, we say $f\colon X\to
Y$ is {\em $\mathcal S$-open} provided $f(X\backslash S)\subset Y$
is open for every $S\in\mathcal S$. Theorem 1.1 yields next
corollary.

\begin{cor}
Let $X$ be a space possessing a binary normal closed subbase
$\mathcal S$. Then every $\mathcal S$-convex $\mathcal S$-open
surjection $f\colon X\to Y$ is invertible.
\end{cor}

Another corollary of Theorem 1.1 is a specification of Ivanov's
results \cite{i3} (see also \cite{i1} and \cite{i2}). Here, a map
$f\colon X\to Y$ is $\mathcal{A}$-soft, where $\mathcal A$ is a
class of spaces, if for any $Z\in\mathcal A$, its closed subset $A$
and any two maps $k\colon Z\to Y$, $h\colon A\to X$ with $f\circ
h=k|A$ there exists a map $g\colon Z\to X$ extending $h$ such that
$f\circ g=k$. When $\mathcal A$ is the family of all
($0$-dimensional) paracompact spaces, then $\mathcal A$-soft maps
are called ($0$-)soft \cite{sc1}.

\begin{cor}
Let $\mathcal A$ be a given class of spaces and $X$ be an absolute extensor for all
$Z\in\mathcal A$. If $X$ has a binary normal closed subbase
$\mathcal S$, then any $\mathcal S$-convex $\mathcal S$-open
surjection $f\colon X\to Y$ is $\mathcal A$-soft.
\end{cor}

Theorem 1.1 is also applied to establish the following proposition:
\begin{pro}
Let $X$ be a $\kappa$-metrizable (resp., $\kappa$-metrizable and connected) compactum with
a normal binary closed subbase $\mathcal S$.
Then every open $\mathcal S$-convex surjection $f\colon X\to Y$ is a zero-soft (resp., soft) map.
\end{pro}

\begin{cor}\cite{i1},\cite{i2}
Let $X$ be a $\kappa$-metrizable (resp., $\kappa$-metrizable and connected) compactum.
Then $\lambda f\colon\lambda X\to \lambda Y$ is a zero-soft (resp., soft) map
for any open surjection $f\colon X\to Y$.
\end{cor}


\section{Proof of Theorem 1.1 and Corollaries 1.2 - 1.4}

Recall  that a set-valued map $\Phi:Z\to X$ is lower semi-continuous
(br., lsc) if the set $\{z\in Z:\Phi(z)\cap U\neq\varnothing\}$ is
open in $Z$ for any open $U\subset X$. $\Phi$ is upper
semi-continuous (br., usc) provided that  the set $\{z\in
Z:\Phi(z)\subset U\}$ is open in $Z$ whenever $U\subset X$ is open.
Upper semi-continuous and compact-valued maps are called usco maps.
If $\Phi$ is both lsc and usc, it is said to be continuous.
Obviously, every continuous set-valued map $\Phi:Z\to X$ is
$\mathcal S$-continuous, where $\mathcal S$ is a binary closed
normal subbase for $X$. Let $C(X,Y)$ denote the set of all
(continuous single-valued) maps from $X$ to $Y$.

{\em Proof of Theorem $1.1.$} Suppose $X$ has a binary normal closed
subbase $\mathcal S$ and $\Phi\colon Z\to X$ is a set-valued
$\mathcal S$-continuous map with closed $\mathcal S$-convex values.
Let $A\subset Z$ be a closed set such that every $f\in C(A,X)$ can
be extended to a map $\bar{f}\in C(Z,X)$. Fix a selection $g\in
C(A,X)$ for $\Phi|A$ and its extension $\bar{g}\in C(Z,X)$. By
\cite[Theorem 1.5.18]{vm}, there exists a (continuous) map
$\xi:X\times\exp X\to X$, defined by
$$\xi(x,F)=\bigcap\{I_{\mathcal S}(\{x,a\}):a\in F\}\cap I_{\mathcal
S}(F),$$ where $\exp X$ is the space of all closed subsets of $X$
with the Vietoris topology. This map has the following properties
for any $F\in\exp X$: (i) $\xi(x,F)=x$ if $x\in I_{\mathcal S}(F)$;
(ii) $\xi(x,F)\in I_{\mathcal S}(F)$, $x\in X$. Because each
$\Phi(z)$, $z\in Z$, is a closed $\mathcal S$-convex set,
$I_{\mathcal S}(\Phi(z))=\Phi(z)$, see \cite[Theorem 1.5.7]{vm}. So,
for all $z\in Z$ we have
$h(z)=\xi\big(\bar{g}(z),\Phi(z)\big)\in\Phi(z)$. Therefore, we
obtain a map $h\colon Z\to X$ which is a selection for $\Phi$ and
$h(z)=g(z)$ for all $z\in A$. It remains to show that $h$ is
continuous. We can show that the subbase could be supposed to be invariant
with respect to finite intersections. Because $\xi$ is continuous, this would
imply continuity of $h$. But instead of that, we follow the arguments from
the proof of \cite[Theorem 1.5.18]{vm}.

Let $z_0\in Z$ and $x_0=h(z_0)\in W$ with $W$ being open in $X$.
We may assume that $W=X\backslash S$ for
some $S\in\mathcal S$. Because
$x_0$ is the intersection of a subfamily of the binary family $\mathcal S$,
there exists $S^*\in\mathcal S$ containing $x_0$ and disjoint from $S$.
Since $\mathcal S$ is normal, there exist
$S_0, S_1\in\mathcal S$ such that $S\subset S_1\backslash S_0$,
$x_0\in S^*\subset S_0\backslash S_1$ and $S_0\cup S_1=X$. Hence, $x_0\in
(X\backslash S_1)\cap\Phi(z_0)$. Because $\Phi$ is $\mathcal
S$-continuous, there exists a neighborhood $O_1(z_0)$ of $z_0$ such
that $\Phi(z)\cap (X\backslash S_1)\neq\varnothing$ for every $z\in
O_1(z_0)$. Observe that $\bar{g}(z_0)\in X\backslash S_1$ provided
$\Phi(z_0)\cap S_1\neq\varnothing$, otherwise $x_0\in
I_{\mathcal S}(\{\bar{g}(z_0),a\})\subset S_1$, where
$a\in\Phi(z_0)\cap S_1$. Consequently, we have two possibilities:
either $\Phi(z_0)\subset X\backslash S_1$ or $\Phi(z_0)$ intersects
both $S_1$ and $X\backslash S_1$. In the first case there exists a
neighborhood $O_2(z_0)$ with $\Phi(z)\subset X\backslash S_1$ for
all $z\in O_2(z_0)$, and in the second one take $O_2(z_0)$ such that
$\bar{g}\big(O_2(z_0)\big)\subset X\backslash S_1$ (recall that in
this case $\bar{g}(z_0)\in X\backslash S_1$). In both cases let
$O(z_0)=O_1(z_0)\cap O_2(z_0)$. Then, in the first case  we have
$h(z)\in\Phi(z)\subset X\backslash S_1\subset X\backslash S$ for
every $z\in O(z_0)$. In the second case let $a(z)\in\Phi(z)\cap
(X\backslash S_1)$, $z\in O(z_0)$. Consequently, $h(z)\in
I_{\mathcal S}(\{\bar{g}(z),a(z)\})\subset X\backslash S_1\subset
S_0\subset X\backslash S$ for any $z\in O(z_0)$.  Hence, $h$ is
continuous.

When the set $A$ is a point $a$ we define $g(a)$ to be an arbitrary
point in $\Phi(a)$ and $\bar{g}(x)=g(a)$ for all $x\in X$. Then the above
arguments provide a selection for $\Phi$.
\hfill$\square$

{\em Proof of Corollary $1.2.$} Since each $\Psi(z)$ is $\mathcal
S$-convex, by Theorem 1.1 it suffices to show that $\Psi$ is
$\mathcal S$-continuous. To this end, suppose that $F_0\in\mathcal
S$ and $\Psi(z_0)\cap (X\backslash F_0)\neq\varnothing$ for some
$z_0\in Z$. Then $\Phi(z_0)\cap (X\backslash F_0)\neq\varnothing$,
for otherwise $\Phi(z_0)\subset F_0$ and $\Psi(z_0)$, being
intersection of all $F\in\mathcal S$ containing $\Phi(z_0)$, would
be contained in $F_0$. Since $\Phi$ is $\mathcal S$-continuous,
there exists a neighborhood $O(z_0)\subset Z$ of $z_0$ such that
$\Phi(z)\cap (X\backslash F_0)\neq\varnothing$ for all $z\in
O(z_0)$. Consequently, $\Psi(z)\cap (X\backslash
F_0)\neq\varnothing$, $z\in O(z_0)$.

Suppose now that $\Psi(z_0)\subset X\backslash F_0$. Then
$\Psi(z_0)\cap F_0=\varnothing$, so there exists $S_0\in\mathcal S$
with $\Phi(z_0)\subset S_0$ and $S_0\cap F_0=\varnothing$ (recall
that $\mathcal S$ is binary). Since $\mathcal S$ is normal, we can
find $S_1, F_1\in\mathcal S$ such that $S_0\subset S_1\backslash
F_1$, $F_0\subset F_1\backslash S_1$ and $F_1\cup S_1=X$. Using
again that $\Phi$ is $\mathcal S$-continuous to choose a
neighborhood $U(z_0)\subset Z$ of $z_0$ with $\Phi(z)\subset
X\backslash F_1\subset S_1$ for all $z\in U(z_0)$. Hence,
$\Psi(z)\subset S_1\subset X\backslash F_0$, $z\in U(z_0)$, which
completes the proof. \hfill$\square$

{\em Proof of Corollary $1.3.$} Let $X$ possess a binary normal
closed subbase $\mathcal S$, $f\colon X\to Y$ be an $\mathcal
S$-open $\mathcal S$-convex surjection, and $g:Z\to Y$ be a map.
Since $f$ is both $\mathcal S$-open and closed (recall that $X$ is
compact as a space with a binary closed subbase), the map
$\phi\colon Y\to X$, $\phi(y)=f^{-1}(y)$, is $\mathcal S$-continuous
and $\mathcal S$-convex valued. So is the map $\Phi=\phi\circ g:Z\to
X$. Then, by Theorem 1.1, $\Phi$ admits a continuous selection
$h:Z\to X$. Obviously, $g=f\circ h$. Hence, $f$ is
invertible.\hfill$\square$

{\em Proof of Corollary $1.4.$} Suppose $X$ is a compactum with a
normal binary closed subbase $\mathcal S$ such that $X$ is an
absolute extensor for all $Z\in\mathcal A$. Let us show that every
$\mathcal S$-open $\mathcal S$-convex surjection $f\colon X\to Y$ is
$\mathcal A$-soft. Take a space $Z\in\mathcal A$, its closed subset
$A$ and two maps $k\colon Z\to Y$, $h\colon A\to X$ such that
$k|A=f\circ h$. Then $h$ can be continuously extended to a map
$\bar{h}\colon Z\to X$. Moreover, the set-valued map $\Phi:Z\to X$,
$\Phi(z)=f^{-1}(k(z))$, is $\mathcal S$-continuous and has $\mathcal
S$-convex values. Hence, by Theorem 1.1, there is a selection
$g\colon Z\to X$ for $\Phi$ extending $h$. Then $f\circ g=k$. So,
$f$ is $\mathcal A$-soft. \hfill$\square$

\section{Proof of Proposition 1.5 and Corollary 1.6}

{\em Proof of Proposition $1.5.$}
According to Corollary 1.4, it suffices to show that $X$ is a Dugundji space (resp., an absolute retract)
provided $X$ is a $\kappa$-metrizable (resp., $\kappa$-metrizable and connected) compactum with
a normal binary closed subbase $\mathcal S$ (recall that the class of Dugundji spaces coincides with the class
of compact absolute extensors for $0$-dimensional spaces, see \cite{h}). To this end, we follow the arguments
from the proof of \cite[Proposition 3.2]{vv}. Suppose first that
$X$ is a $\kappa$-metrizable compactum with a normal binary closed subbase $\mathcal S$. Consider $X$ as a
subset of a Tyhonoff cube $\mathbb I^\tau$. Then, by \cite{shi}
(see also \cite{vv} for another proof), there exists a function $\mathrm{e}\colon\mathcal T_X\to\mathcal T_{\mathbb{I}^\tau}$
between
the topologies of $X$ and $\mathbb I^\tau$ such that:
\begin{itemize}
\item[($\mathrm{e}$1)] $\mathrm{e}(\varnothing)=\varnothing$ and $\mathrm{e}(U)\cap X=U$ for any open $U\subset X$;
\item[($\mathrm{e}$2)] $\mathrm{e}(U)\cap\mathrm{e}(V)=\varnothing$ for any two disjoint open sets $U,V\subset X$.
\end{itemize}
Consider the set valued map $r\colon\mathbb I^\tau\to X$ defined by
$$r(y)=\bigcap\{I_{\mathcal S}(\overline{U}):y\in e(U), U\in\mathcal
T_{X}\}\hbox{~}\mbox{if}\hbox{~} y\in\bigcup\{e(U): U\in\mathcal
T_{X}\}\leqno{(1)}$$
$$\hbox{~}\mbox{and}\hbox{~} r(y)=X
\hbox{~}\mbox{otherwise},\hbox{~}$$ where $\overline U$ is the
closure of $U$ in $X$. According to condition ($\mathrm{e}$2),
the system
$\gamma_y=\{U\in\mathcal T_{X} :y\in e(U)\}$ is linked for every $y\in\mathbb I^\tau$.
Consequently, $\omega_y=\{S\in\mathcal S:\overline{U}\subset
S\hbox{~}\mbox{for some}\hbox{~}U\in\gamma_y\}$ is also linked. This implies
$r(y)=\bigcap\{S:S\in\omega_y\}\neq\varnothing$ because $\mathcal S$ is binary.

{\em Claim. $r(x)=\{x\}$ for every $x\in X$.}

Suppose there is another point $z\in r(x)$. Then, by normality of $\mathcal S$, there exist
two elements $S_0,S_1\in\mathcal S$ such that $x\in S_0\backslash S_1$, $z\in S_1\backslash S_0$
and $S_0\cup S_1=X$. Choose an open neighborhood $V$ of $x$ with $\overline V\subset S_0\backslash S_1$.
Observe that $x\in\mathrm{e}(V)$, so $z\in I_{\mathcal S}(\overline{V})\subset S_0$, a contradiction.

Finally, we can show that $r$ is upper semi-continuous. Indeed, let $r(y)\subset W$ with $y\in\mathbb I^\tau$ and
$W\in\mathcal T_X$. Then there exist finitely many $U_i\in\mathcal
T_{X}$, $i=1,2,..,k$, such that $y\in \bigcap_{i=1}^{i=k}e(U_i)$ and
$\bigcap_{i=1}^{i=k}I_{\mathcal S}(\overline{U}_i)\subset W$.
Obviously, $r(y')\subset W$ for all
$y'\in\bigcap_{i=1}^{i=k}e(U_i)$. So, $r$ is an usco retraction from
$\mathbb I^\tau$ onto $X$. According to \cite{dr}, $X$ is a Dugundji
space.

Suppose now, that $X$ is connected. By \cite{vm}, any set of the form $I_{\mathcal
S}(F)$ is $\mathcal S$-convex, so is
each $r(y)$. According to
\cite[Corollary 1.5.8]{vm}, all closed $\mathcal S$-convex subsets
of $X$ are also connected. Hence, the map $r$, defined by $(1)$, is
connected-valued. Consequently, by \cite{dr}, $X$ is an absolute
extensor in dimension 1, and there exists a map $r_1:\mathbb I^\tau
\to \exp X$ with $r_1(x)=\{x\}$ for all $x\in X$, see \cite[Theorem
3.2]{f}. On the other hand, since $X$ is normally
supercompact, there exists a retraction $r_2$ from $\exp X$ onto
$X$, see \cite[Corollary 1.5.20]{vm}. Then the composition $r_2\circ
r_1\colon\mathbb I^\tau\to X$ is a (single-valued) retraction. So,
$X\in AR$. $\Box$

{\em Proof of Corollary $1.6.$}
It is well known that $\lambda$ is a continuous functor preserving
open maps, see \cite{f1}. So, $\lambda X$ is $\kappa$-metrizable. Moreover, $\lambda X$
is connected if so is $X$. On the other hand,
the family $\mathcal S=\{F^+: F\hbox{~}\mbox{is closed in}\hbox{~}X\}$, where
$F^+=\{\eta\in\lambda X: F\in\eta\}$, is a binary normal subbase for $\lambda X$.
Observe that $\lambda f$ is $\mathcal S$-convex because
$(\lambda f)^{-1}(\nu)=\bigcap\{f^{-1}(H)^+:H\in\nu\}$ for every $\nu\in\lambda Y$.
Then, Proposition 1.5 completes the proof. $\Box$

The next proposition shows that the statements from Proposition 1.5 and Corollary 1.6
are actually equivalent. At the same time it provides more information about validity of Corollary 1.4.
\begin{pro}
For any class $\mathcal A$  the following statements are equivalent:
\begin{itemize}
\item[(i)] If $X$ is a 
compactum possessing a normal binary closed subbase $\mathcal S$,
then any open $\mathcal S$-convex surjection $f\colon X\to Y$ is
$\mathcal A$-soft. 
\item[(ii)] The map $\lambda f\colon\lambda X\to \lambda Y$ is $\mathcal A$-soft
for any compactum $X$ 
and any open surjection $f\colon X\to Y$.
\end{itemize}
\end{pro}

\begin{proof}
$(i)\Rightarrow (ii)$ Let $X$ be a
compactum and  $f\colon X\to Y$ be an open surjection. It
is easily seen that
$\lambda f$ is an open surjection too. 
We already noted that $\mathcal S=\{F^+:F\subset
X\hbox{~}\mbox{is closed}\}$ is a normal binary closed subbase for
$\lambda X$ and $\lambda f$ is a $\mathcal
S$-convex and  open map. Hence, by (i), $\lambda f$ is $\mathcal A$-soft.

\noindent $(ii)\Rightarrow (i)$. Suppose $X$ is a
compactum possessing a normal binary closed subbase $\mathcal S$,
and $f\colon X\to Y$ is an $\mathcal S$-convex open surjection. To
show that $f$ is $\mathcal A$-soft, take a space $Z\in\mathcal A$, its
closed subset $A$ and two maps $h:A\to X$, $g\colon Z\to Y$ with
$f\circ h=g|A$. So, we have the following diagram, where $i_X$ and
$i_Y$ are embeddings defined by $x\to\eta_x$ and $y\to\eta_y$,
respectively.

$$
\begin{CD}
A@>{h}>>X@>{i_X}>>\lambda X\cr @V{id}VV @VV{f}V @VV{\lambda f}V\cr
Z@>{g}>>Y@>{i_Y}>>\lambda Y\cr
\end{CD}
$$

\medskip\noindent Since, by (ii), $\lambda f$ is $\mathcal A$-soft, there
exists a map $g_1\colon Z\to\lambda X$ such that $h=g_1|A$ and
$\lambda f\circ g_1=g$. The last equality implies that
$g_1(Z)\subset (\lambda f)^{-1}(Y)$. According to \cite[Corollary
2.3.7]{vm}, there exists a retraction $r\colon\lambda X\to X$,
defined by $$r(\eta)=\bigcap\{F\in\mathcal
S:F\in\eta\}.\leqno{(2)}$$ Consider now the map $\bar{g}=r\circ
g_1\colon Z\to X$. Obviously, $\bar{g}$ extends $h$. Let us show
that $f\circ\bar{g}=g$. Indeed, for any $z\in Z$ we have $$g_1(z)\in
(\lambda f)^{-1}(g(z))=(f^{-1}(g(z)))^+.$$ Since $f$ is $\mathcal
S$-convex, $I_{\mathcal S}\big(f^{-1}(g(z))\big)=f^{-1}(g(z))$, see
\cite[Theorem 1.5.7]{vm}. Hence, $f^{-1}(g(z))$ is the intersection
of the family $\{F\in\mathcal S:f^{-1}(g(z))\subset F\}$ whose
elements belong to any $\eta\in (\lambda f)^{-1}(g(z))$. It follows
from $(2)$ that $r(\eta)\in f^{-1}(g(z))$, $\eta\in (\lambda
f)^{-1}(g(z))$. In particular, $\bar{g}(z)\in f^{-1}(g(z))$.
Therefore, $f\circ\bar{g}=g$.
\end{proof}

The following corollary follows from Corollary 1.4 and Proposition 3.1.

\begin{cor}
If $X$ is a compactum with a binary normal closed subbase $\mathcal S$
such that $\lambda X$ is an absolute extensor for a given class $\mathcal A$, then
any open $\mathcal S$-convex surjection $f\colon X\to Y$ is
$\mathcal A$-soft.
\end{cor}

\textbf{Acknowledgments.} The author would like to express his
gratitude to M. Choban for his continuous support and valuable remarks.


\end{document}